\documentclass[a4paper]{amsart}
\usepackage{amsmath, amssymb, graphicx}
\usepackage{hyperref}
\numberwithin{equation}{section}
\newtheorem{theorem}{Theorem}[section]
\newtheorem{corollary}{Corollary}[section]
\newtheorem{definition}{Definition}[section]
\newtheorem{lemma}{Lemma}[section]

\theoremstyle{remark}

\title[Some inequalities for a certain subclass of starlike functions]
{Some inequalities for a certain subclass of starlike functions}



\begin{document}
\begin{abstract}
In 2011, Sok\'{o}{\l} (Comput. Math. Appl. 62, 611--619) introduced and studied the class $\mathcal{SK}(\alpha)$ as a certain subclass of starlike functions, consists of all functions $f$ ($f(0)=0=f'(0)-1$) which satisfy in the following subordination relation:
\begin{equation*}
\frac{zf'(z)}{f(z)}\prec \frac{3}{3+(\alpha-3)z-\alpha z^2} \qquad |z|<1,
\end{equation*}
where $-3<\alpha\leq1$. Also, he obtained some interesting results for the class $\mathcal{SK}(\alpha)$. In this paper, some another properties of this class, including infimum of $\mathfrak{Re}\frac{f(z)}{z}$, order of strongly starlikeness, the sharp logarithmic coefficients inequality and the sharp Fekete-Szeg\"{o} inequality are investigated.
\end{abstract}

\author[ R. Kargar, H. Mahzoon, N. Kanzi] {R. Kargar, H. Mahzoon and N. Kanzi}
\address{Young Researchers and Elite Club,
Urmia Branch, Islamic Azad University, Urmia, Iran} \email
{rkargar1983@gmail.com {\rm (Rahim Kargar)}}
\address{ Department of Mathematics, Islamic Azad University, Firoozkouh
Branch, Firoozkouh, Iran.}
\email{mahzoon$_{-}$hesam@yahoo.com {\rm (Hesam Mahzoon)}}
\address{Department of Mathematics, Payame Noor University, Tehran, Iran}
       \email {nad.kanzi@gmail.com {\rm (Nader Kanzi)}}

\maketitle

\section{Introduction}

Let $\mathcal{A}$ denote the class of functions $f(z)$ of the
form:
\begin{equation}\label{f}
    f(z)=z+ a_2 z^2+\cdots+a_n z^n+\cdots,
\end{equation}
which are analytic and normalized in the unit disk $\Delta=\{z\in \mathbb{C} :
|z|<1\}$. The subclass of
$\mathcal{A}$ consisting of all univalent functions $f(z)$ in
$\Delta$ is denoted by $\mathcal{S}$. A function $f\in\mathcal{S}$
is called starlike (with respect to $0$), denoted by
$f\in\mathcal{S}^*$, if $tw\in f(\Delta)$ whenever $w\in f(\Delta)$
and $t\in[0, 1]$. The class $\mathcal{S}^*(\gamma)$ of starlike functions of order $\gamma\leq1$, is defined by
\begin{equation*}
   \mathcal{S}^*(\gamma):=\left\{ f\in \mathcal{A}:\ \ {\rm Re} \left(\frac{zf'(z)}{f(z)}\right)> \gamma, \ z\in
    \Delta\right\}.
\end{equation*}
Note that if $0\leq \gamma<1$, then $\mathcal{S}^*(\gamma)\subset\mathcal{S}$. Moreover, if $\gamma<0$, then the function $f$ may fail to be univalent.
A function $f\in\mathcal{S}$ that maps $\Delta$
onto a convex domain, denoted by $f\in\mathcal{K}$, is called a
convex function.
Also, the class $\mathcal{K}(\gamma)$ of convex functions of order $\gamma\leq1$,
is defined by
\begin{equation*}
   \mathcal{K}(\gamma):=\left\{ f\in \mathcal{A}:\ \ {\rm Re} \left(1+\frac{zf''(z)}{f'(z)}\right)> \gamma, \ z\in
    \Delta\right\}.
\end{equation*}
 In particular we denote
$\mathcal{S}^*(0)\equiv \mathcal{S}^*$ and $\mathcal{K}(0)\equiv\mathcal{K}$. The classes $\mathcal{S}^*(\gamma)$ and $\mathcal{K}(\gamma)$ introduced by Robertson (see \cite{ROB}). Also, as usual, let
\begin{equation*}
  \mathcal{S}^*_t(\gamma):=\left\{f\in \mathcal{A}: \left|\arg \frac{zf'(z)}{f(z)} \right|<\frac{\pi \gamma}{2}, z\in \Delta\right\},
\end{equation*}
be the class of strongly starlike functions of order $\gamma$ ($0<\gamma\leq 1$) (see \cite{St}). We note that $\mathcal{S}^*_t(\gamma)\subset \mathcal{S}^*$ for $0<\gamma<1$ and $\mathcal{S}^*_t(1)\equiv \mathcal{S}^*$.
Define by $\mathcal{Q}(\gamma)$, the class of all functions $f\in \mathcal{A}$ so that satisfy the condition
\begin{equation}\label{Q(gamma)}
  {\rm Re}\left(\frac{f(z)}{z}\right)>\gamma\qquad (0\leq \gamma<1).
\end{equation}
We denote by $\mathfrak{B}$ the class of analytic functions $w(z)$ in
$\Delta$ with $w(0) = 0$ and $|w(z)| < 1$, $(z \in \Delta)$.
If $f$ and $g$ are two of the functions in $\mathcal{A}$, we say
that $f$ is subordinate to $g$, written $f (z)\prec g(z)$, if there
exists a $w\in\mathfrak{B}$ such that $f (z) = g(w(z))$, for all
$z\in\Delta$.

Furthermore, if the function $g$ is univalent in $\Delta$, then
we have the following equivalence:
\begin{equation*}
    f (z)\prec g(z) \Leftrightarrow (f (0) = g(0)\quad {\rm and}\quad f (\Delta)\subset g(\Delta)).
\end{equation*}
Also $|w(z)|\leq |z|$, by Schwarz's lemma and therefore
\begin{equation*}
  \{f(z):|z|<r\}\subset\{g(z):|z|<r\}\qquad(0<r<1).
\end{equation*}
It follows that
\begin{equation*}
  \max_{|z|\leq r}|f(z)|\leq \max_{|z|\leq r}|g(z)|\qquad(0<r<1).
\end{equation*}

We now recall from \cite{JSOK}, a one-parameter family of functions as follows:
\begin{equation}\label{pb}
  p_b(z):=\frac{1}{1-(1+b)z+bz^2}\qquad (z\in\Delta).
\end{equation}
We note that if $|b|<1$, then
\begin{equation*}
  {\rm Re}\{p_b(z)\}>\frac{1-3b}{2(1-b)^2},
\end{equation*}
and if $b\in[-1/3, 1)$, then
\begin{equation*}
  \frac{1-3b}{2(1-b)^2}<{\rm Re}\{p_b(e^{i\varphi})\}\leq \frac{1}{2(1+b)}=p_b(-1)\qquad (0<\varphi<2\pi).
\end{equation*}
Also, if $b\in[-1/3, 1]$, then the function $p_b$ defined in \eqref{pb} is univalent in $\Delta$ and has no loops when $-1/3\leq b<1 $.

By putting $b=-\alpha/3$ in the function \eqref{pb}, we have:
\begin{equation}\label{qalpha}
  \widetilde{q}_\alpha(z):=\frac{3}{3+(\alpha-3)z-\alpha z^2}\qquad (z\in\Delta).
\end{equation}
The function $\widetilde{q}_\alpha(z)$ is univalent in $\Delta$ when $\alpha\in (-3, 1]$ (see Figure 1 for $\alpha=1$).
\begin{figure}[htp]
 \centering
 \includegraphics[width=6cm]{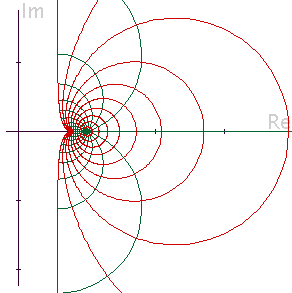}\\
 \caption{The graph of $\widetilde{q}_\alpha(\Delta)$ for $\alpha=1$}\label{Fig:1}
\end{figure}

 Note that
\begin{equation}\label{sumq}
  \widetilde{q}_\alpha(z)=1+\sum_{n=1}^{\infty }\mathcal{B}_nz^n,
\end{equation}
where
\begin{equation}\label{Bn}
  \mathcal{B}_n=\frac{3}{3+\alpha}\left[1+(-1)^n(\alpha/3)^{n+1}\right]\qquad(n=1,2,\ldots).
\end{equation}
Over the years, the definition of a certain subclass of analytic functions by using the subordination relation has been investigated by many works including (for example) \cite{KESBooth}, \cite{KES(Siberian)}, \cite{KES(Complex)}, \cite{PSahoo}, \cite{PKfin}, \cite {SahooSharma} and \cite{HM1}.
We now recall from \cite{JSOK}, the following definition which is used from subordination.
\begin{definition}\label{Def}
  The function $f\in\mathcal{A}$ belongs to the class $\mathcal{SK}(\alpha)$, $\alpha\in (-3, 1]$, if it satisfies the condition
  \begin{equation}\label{defsok}
    \frac{zf'(z)}{f(z)}\prec \widetilde{q}_\alpha(z)\qquad (z\in\Delta),
  \end{equation}
  where $\widetilde{q}_\alpha$ is given by \eqref{qalpha}.
\end{definition}
Since ${\rm Re} \{\widetilde{q}_\alpha(z)\}>9(1+\alpha)/2(3+\alpha)^2$, therefore if $f\in\mathcal{SK}(\alpha)$, then
\begin{equation}\label{reSK}
  {\rm Re}\left(\frac{zf'(z)}{f(z)}\right)>\frac{9(1+\alpha)}{2(3+\alpha)^2}\qquad (z\in\Delta).
\end{equation}
This means that if $f\in\mathcal{SK}(\alpha)$, then it
is starlike of order $\gamma$  where $\gamma=9(1+\alpha)/2(3+\alpha)^2$. Also, $\mathcal{SK}(\alpha)\subset \mathcal{S}^*$ when $-1 \leq \alpha < 1$, $\mathcal{SK}(0)\equiv \mathcal{S}^*(1/2)$,
$\mathcal{SK}(1)\equiv \mathcal{S}^*(9/16)$ and $\mathcal{SK}(-1)\equiv \mathcal{S}^*$.

We denote by $\mathcal{P}$ the well-known class of analytic functions $p(z)$ with
$p(0) = 1$ and ${\rm Re}(p(z))>0$, $z\in\Delta$.

For the proof of our results, we need the following Lemmas.

\begin{lemma}\label{MIMO}
\cite[p.35]{MM-book} Let $ \Xi$ be a set in the complex plane
$\mathbb{C}$ and let b be a complex number such that
${\rm Re}(b) > 0$. Suppose that a function $
\psi:\mathbb{C}^2\times \Delta\rightarrow \mathbb{C}$ satisfies the
condition:
\begin{equation*}
   \psi(i\rho,\sigma;z)\not\in
   \Xi,
\end{equation*}
for all real $\rho,\sigma\leq-\mid
b-i\rho\mid^2/(2{\rm Re}b)$ and all $z\in\Delta$. If the
function $p(z)$ defined by $p(z) =b+b_1z+b_2z^2+\cdots$ is analytic
in $\Delta$ and if
\begin{equation*}
   \psi(p(z),zp'(z);z)\in \Xi,
\end{equation*}
then ${\rm Re} (p(z))>0$ in $\Delta$.
\end{lemma}
\begin{lemma}\label{FEK}\cite{KM}
Let the function $g(z)$ given by
\begin{equation*}
 g(z)=1+c_1z+c_2z^2+\cdots,
\end{equation*}
be in the class $\mathcal{P}$. Then, for any complex number $\mu$
\begin{equation*}
 |c_2-\mu c_1^2|\leq 2\max\{1,|2\mu-1|\}.
\end{equation*}
The result is sharp.
\end{lemma}
The structure of the paper is the following. In Section \ref{sec1}, at first, we obtain a lower bound for the ${\rm Re}\frac{f(z)}{z}$ and by using it we get $\mathcal{S}^*\subset \mathcal{Q}(1/2)$. In the sequel, we obtain the order of strongly starlikeness for the functions which belong to the class $\mathcal{SK}(\alpha)$. In Section \ref{sec2}, sharp coefficient logarithmic inequality and sharp Fekete-Szeg\"{o} inequality are obtained.

\setcounter{theorem}{0}

\section{\large Main results}\label{sec1}
\noindent
The first result is the following. By using the Theorem \ref{Th. Re z/f(z)} (bellow), we get the well-known result about the starlike univalent functions (Corollary \ref{cro21}).
\begin{theorem}\label{Th. Re z/f(z)}
  Let $f\in \mathcal{A}$ be in the class $\mathcal{SK}(\alpha)$ and $0\leq \alpha\leq 1$. Then
  \begin{equation}\label{Rez/fz}
    {\rm Re}\left(\frac{f(z)}{z}\right)>\gamma(\alpha):=\frac{2\alpha^2+3\alpha+9}{3(\alpha^2+3\alpha+6)}\qquad (z\in\Delta).
  \end{equation}
  That is means that $\mathcal{SK}(\alpha)\subset\mathcal{Q}(\gamma(\alpha))$.
\end{theorem}
\begin{proof}
  Put $\gamma(\alpha):=\gamma$. Thus $0<\gamma\leq 1/2$ when $0\leq \alpha\leq 1$. Let $p$ be defined by
  \begin{equation}\label{p}
    p(z)=\frac{1}{1-\gamma}\left(\frac{f(z)}{z}-\gamma\right)\qquad (z\in\Delta).
  \end{equation}
  Then $p$ is analytic in $\Delta$, $p(0) = 1$ and
  \begin{equation*}
    \frac{zf'(z)}{f(z)}=1+\frac{(1-\gamma)zp'(z)}{(1-\gamma)p(z)+\gamma}=\psi(p(z),zp'(z)),
  \end{equation*}
  where
  \begin{equation}\label{psi}
    \psi(a,b):=1+\frac{(1-\gamma)b}{(1-\gamma)a+\gamma}.
  \end{equation}
 By \eqref{reSK}, we define $\Omega_{\alpha'}$ as follows:
 \begin{equation*}
   \{\psi(p(z),zp'(z)): z\in\Delta\}\subset \{w\in \mathbb{C}: {\rm Re} \{w\}>\alpha'\}=:\Omega_{\alpha'},
 \end{equation*}
  where $\alpha'=9(1+\alpha)/2(3+\alpha)^2$. For all real $\rho$ and $\sigma$, which $\sigma\leq -\frac{1}{2}(1+\rho^2)$, we have
  \begin{align*}
    {\rm Re}\{\psi(i\rho,\sigma)\} &= {\rm Re}\left\{1+\frac{(1-\gamma)\sigma}{(1-\gamma)i\rho+\gamma}\right\}
    =1+\frac{\gamma(1-\gamma)\sigma}{(1-\gamma)^2 \rho^2+\gamma^2} \\
     &\leq 1-\frac{1}{2}\gamma(1-\gamma)\frac{1+\rho^2}{(1-\gamma)^2 \rho^2+\gamma^2}.
  \end{align*}
  Define
  \begin{equation}\label{h}
    h(\rho)=\frac{1+\rho^2}{(1-\gamma)^2 \rho^2+\gamma^2}.
  \end{equation}
  Then $h'(\rho)=0$ occurs at only $\rho=0$ and we get $h(0)=1/\gamma^2$ and
  \begin{equation*}
    \lim_{\rho\rightarrow\infty}h(\rho)=\frac{1}{(1-\gamma)^2}.
  \end{equation*}
  Since $0<\gamma\leq 1/2$, we have
  \begin{equation*}
    \frac{1}{(1-\gamma)^2}<h(\rho)\leq \frac{1}{\gamma^2}.
  \end{equation*}
  Therefore
  \begin{equation*}
    {\rm Re}\{\psi(i\rho,\sigma)\}\leq 1-\frac{1}{2}\gamma(1-\gamma)\frac{1}{(1-\gamma)^2}=\frac{2-3\gamma}{2(1-\gamma)}=:\alpha'.
  \end{equation*}
  This shows that ${\rm Re}\{\psi(i\rho,\sigma)\}\not\in \Omega_{\alpha'}$. Applying Lemma \ref{MIMO}, we get ${\rm Re}p(z)>0$ in $\Delta$,
  and this shows that the inequality of \eqref{Rez/fz} holds. This proves the theorem.
\end{proof}
Setting $\alpha=0$ in the Theorem \ref{Th. Re z/f(z)}, we get the following well known result:
\begin{corollary}\label{cro21}
  Let $f\in \mathcal{A}$ be defined by \eqref{f}. Then $\mathcal{S}^*\subset \mathcal{Q}(1/2)$, i.e.
  \begin{equation*}
    {\rm Re}\left(\frac{zf'(z)}{f(z)}\right)>0\Rightarrow {\rm Re}\left(\frac{f(z)}{z}\right)>\frac{1}{2}\qquad (z\in \Delta).
  \end{equation*}
\end{corollary}
We remark that in \cite{KESmalaya, KESij}, the authors with a different method have shown that $\mathcal{S}^*\subset \mathcal{Q}(1/2)$.

By putting $z=re^{i\varphi}(r<1)$, $\varphi\in [0,2\pi)$ and with a simple calculation, we have
\begin{align*}
  \widetilde{q}_\alpha(re^{i\varphi}) &=\frac{3}{3+(\alpha-3)re^{i\varphi}-\alpha r^2e^{2i\varphi}}  \\
   &= \frac{9+3(\alpha-3)r\cos\varphi-3\alpha r^2\cos2\varphi-3i[(\alpha-3)r\sin\varphi-\alpha r^2\sin2\varphi]}
   {9+(\alpha-3)^2r^2+\alpha^2 r^4+2r(\alpha-3)(3-\alpha r^2)\cos \varphi-6\alpha r^2\cos2\varphi}.
\end{align*}
Hence
\begin{align}\label{phir}
  \left|\frac{{\rm Im}\{\widetilde{q}_\alpha(re^{i\varphi})\}}{{\rm Re}\{\widetilde{q}_\alpha(re^{i\varphi})\}}\right| &=
  \left|\frac{3[(\alpha-3)r\sin\varphi-\alpha r^2\sin2\varphi]}{9+3(\alpha-3)r\cos\varphi-3\alpha r^2\cos2\varphi}\right| \nonumber\\
   &< \frac{3(3-\alpha)r+|\alpha|r^2}{9-3(3-\alpha)r-3|\alpha|r^2}=:\phi(r)\qquad (r<1, \ \ -3<\alpha\leq1).
\end{align}
For such $r$ the curve $\widetilde{q}_\alpha(re^{i\varphi})$, $\varphi\in[0,2\pi)$, has no loops and $\widetilde{q}_\alpha(re^{i\varphi})$ is univalent in
$\Delta_r =\{z : |z| < r\}$. Therefore
\begin{equation}\label{[]Leftrightarrow[]}
  \left[\frac{zf'(z)}{f(z)}\prec \widetilde{q}_\alpha(z),\ \ z\in \Delta_r\right]\Leftrightarrow
  \left[\frac{zf'(z)}{f(z)}\in\widetilde{q}_\alpha(\Delta_r), \ \ z\in \Delta_r\right].
\end{equation}
The above relations give the following theorem.
\begin{theorem}\label{th.arg}
  Let $-1<\alpha\leq1$. If $f\in \mathcal{SK}(\alpha)$, then $f$ is strongly starlike of order
\begin{equation*}
\frac{2}{\pi} \arctan \left\{\frac{3(3-\alpha)+|\alpha|}{9-3(3-\alpha)-3|\alpha|}\right\},
\end{equation*}
in the unit disc $\Delta$.
\end{theorem}
\begin{proof}
  Since ${\rm Re}\{zf'(z)/f(z)\}>0$ in the unit disk, and from \eqref{phir} and \eqref{[]Leftrightarrow[]}, we have
  \begin{align*}
    \left|\arg\left\{\frac{zf'(z)}{f(z)}\right\}\right| &= \left|\arctan \frac{{\rm Im}(zf'(z)/f(z))}{{\rm Re}(zf'(z)/f(z))}\right|\\
     &\leq \left|\arctan \frac{{\rm Im}(\widetilde{q}_\alpha(re^{i\varphi}))}{{\rm Re}(\widetilde{q}_\alpha(re^{i\varphi}))}\right|\\
     &< \arctan\phi(r),
  \end{align*}
  where $\phi(r)$ defined by \eqref{phir}. Now by letting $r\rightarrow 1^-$ the proof of this theorem is completed.
\end{proof}

\section{On coefficients}\label{sec2}
The logarithmic coefficients $\gamma_n$ of $f(z)$ are defined by
\begin{equation}\label{log coef}
  \log\frac{f(z)}{z}=\sum_{n=1}^{\infty}2\gamma_n z^n\qquad (z\in \Delta).
\end{equation}
These coefficients play an important role for various estimates in the theory of univalent functions. For functions in the class $\mathcal{S}^*$, it is easy to prove that $|\gamma_n|\leq 1/n$ for $n\geq1$ and equality holds for the Koebe
function. Here, we get the sharp logarithmic coefficients inequality for the functions which belong to the class $ \mathcal{SK}(\alpha)$. First, we present a subordination relation related with the class $\mathcal{SK}(\alpha)$. This relation is then used to obtain sharp inequality for their logarithmic
coefficients.
\begin{theorem}\label{th. log}
  Let $f\in \mathcal{A}$ and $-3<\alpha\leq 1$. If $f\in \mathcal{SK}(\alpha)$, then there exists a function $w(z)\in \mathfrak{B}$ such that
  \begin{equation}\label{log f/z}
    \log\frac{f(z)}{z}=\int_{0}^{z}\frac{\widetilde{q}_\alpha(w(t))-1}{t}{\rm d}t\qquad(z\in \Delta).
  \end{equation}
\end{theorem}
\begin{proof}
  By Definition \ref{Def}, if $f\in \mathcal{SK}(\alpha)$, then
  \begin{equation*}
    \frac{zf'(z)}{f(z)}\prec\widetilde{q}_\alpha(z)\qquad(z\in \Delta)
  \end{equation*}
  or
  \begin{equation*}
    z\left\{\log \frac{f(z)}{z}\right\}'\prec\widetilde{q}_\alpha(z)-1\qquad(z\in \Delta).
  \end{equation*}
  From the definition of subordination, there exists a function $w(z)\in \mathfrak{B}$ so that
    \begin{equation*}
    z\left\{\log \frac{f(z)}{z}\right\}'=\widetilde{q}_\alpha(w(z))-1\qquad(z\in \Delta).
  \end{equation*}
  Now the assertion follows by integrating of the last equality.
\end{proof}
\begin{corollary}\label{cr1}
    Let $f\in \mathcal{A}$ and $-3<\alpha\leq 1$. If $f\in \mathcal{SK}(\alpha)$, then
  \begin{equation}\label{log f/z}
    \log\frac{f(z)}{z}\prec\int_{0}^{z}\frac{\widetilde{q}_\alpha(t)-1}{t}{\rm d}t\qquad(z\in \Delta).
  \end{equation}
\end{corollary}
 The celebrated de Branges' inequalities (the former Milin conjecture) for
univalent functions $f$ state that
\begin{equation*}
  \sum_{n=1}^{k}(k-n+1)|\gamma_n|^2\leq \sum_{n=1}^{k}\frac{k+1-n}{n}\qquad(k=1,2,\ldots),
\end{equation*}
with equality if and only if $f (z) = e^{-i\theta}k(e^{i\theta}z)$ (see \cite{de Bran}). De Branges \cite{de Bran} used this inequality to prove the celebrated Bieberbach conjecture. Moreover, the de Branges'
inequalities have also been the source of many other interesting inequalities involving
logarithmic coefficients of $f\in \mathcal{S}$ such as (see \cite{duren Leung})
\begin{equation*}
  \sum_{n=1}^{\infty}|\gamma_n|^2\leq \sum_{n=1}^{\infty}\frac{1}{n^2}=\frac{\pi^2}{6}.
\end{equation*}
 Now, we have the following.
\begin{theorem}
  Let $f\in \mathcal{A}$ belongs to the class $\mathcal{SK}(\alpha)$ and $-3<\alpha\leq 1$. Then the logarithmic coefficients of $f$ satisfy in the inequality
  \begin{equation}\label{log ineq}
    \sum_{n=1}^{\infty}|\gamma_n|^2\leq\frac{1}{4(3+\alpha)^2}\left[\frac{3\pi^2}{2}+
    6\alpha Li_2\left(-\alpha/3\right)+\alpha^2 Li_2\left(\alpha^2/9\right)\right],
  \end{equation}
  where $Li_2$ is defined as following
  \begin{equation}\label{LI2}
    Li_2(z)=\sum_{n=1}^{\infty}\frac{z^n}{n^2}=\int_{z}^{0}\frac{\ln (1-t)}{t}{\rm d}t.
  \end{equation}
  The inequality is sharp.
\end{theorem}
\begin{proof}
  Let $f\in \mathcal{SK}(\alpha)$. Then by Corollary \ref{cr1}, we have
    \begin{align}\label{log f/z=A}
    \log\frac{f(z)}{z}\prec\int_{0}^{z}\frac{\widetilde{q}_\alpha(t)-1}{t}{\rm d}t\qquad(z\in \Delta).
  \end{align}
  Again, by using \eqref{log coef} and \eqref{Bn}, the relation \eqref{log f/z=A} implies that
  \begin{equation*}
    \sum_{n=1}^{\infty}2\gamma_n z^n\prec \sum_{n=1}^{\infty}\frac{\mathcal{B}_n}{n}z^n\qquad (z\in \Delta).
  \end{equation*}
  Now by Rogosinski's theorem \cite[Sec. 6.2]{Dur}, we get
  \begin{align*}
    4\sum_{n=1}^{\infty}|\gamma_n|^2 &\leq \sum_{n=1}^{\infty}\frac{1}{n^2}|\mathcal{B}_n|^2\\
    &=\frac{9}{(3+\alpha)^2}\sum_{n=1}^{\infty} \frac{1}{n^2}\left|1+(-1)^n(\alpha/3)^{n+1}\right|^2\\
    &=\frac{9}{(3+\alpha)^2}\left(\frac{\pi^2}{6}+\frac{2\alpha}{3}Li_2\left(\frac{-\alpha}{3}\right)
    +\frac{\alpha^2}{9}Li_2\left(\frac{\alpha^2}{9}\right)\right),
  \end{align*}
  where $Li_2$ is given by \eqref{LI2}. Therefore the desired inequality \eqref{log ineq} follows. For the sharpness of \eqref{log ineq}, consider
  \begin{equation}\label{sharp function}
    \phi_\alpha(z)=z\exp \int_{0}^{z}\frac{\widetilde{q}_\alpha(t)-1}{t}{\rm d}t.
  \end{equation}
  It is easy to see that $\phi_\alpha(z)\in \mathcal{SK}(\alpha)$ and $\gamma_n(\phi_\alpha)=\mathcal{B}_n/2n$, where $B_n$ is given by \eqref{Bn}. Therefore, we have the equality in \eqref{log ineq} and concluding the proof.
\end{proof}
\begin{theorem}
  Let $f\in \mathcal{A}$ be a member of $\mathcal{SK}(\alpha)$. Then the logarithmic coefficients of $f$ satisfy
  \begin{equation*}
    |\gamma_n|\leq \frac{3-\alpha}{6n}\qquad (-3<\alpha\leq 1, n\geq1).
  \end{equation*}
\end{theorem}
\begin{proof}
  Let $f\in \mathcal{SK}(\alpha)$. Then by Definition \ref{Def} we have
  \begin{equation*}
    \frac{zf'(z)}{f(z)}\prec\widetilde{q}_\alpha(z)\qquad(z\in \Delta)
  \end{equation*}
  or
  \begin{equation}\label{z log f}
    z\left\{\log \frac{f(z)}{z}\right\}'\prec\widetilde{q}_\alpha(z)-1\qquad(z\in \Delta).
  \end{equation}
  Applying \eqref{sumq} and \eqref{log coef}, the above subordination relation \eqref{z log f} implies that 
  \begin{equation*}
    \sum_{n=1}^{\infty}2n \gamma_n z^n\prec \sum_{n=1}^{\infty}\mathcal{B}_n z^n.
  \end{equation*}
  Applying the Rogosinski theorem \cite{Rog}, we get the inequality $2n |\gamma_n|\leq |\mathcal{B}_1|=1-\alpha/3$. This completes the proof.
\end{proof}

The problem of finding sharp upper bounds for the coefficient
functional $|a_3-\mu a_2^2|$ for different subclasses of the
normalized analytic function class $\mathcal{A}$ is known as the
Fekete-Szeg\"{o} problem. We recall here that, for a
univalent function $f(z)$ of the form \eqref{f}, the $k$th root
transform is defined by
\begin{equation}\label{F(z)}
 F(z)=[f(z^k)]^{1/k}=z+\sum_{n=1}^{\infty}b_{kn+1}z^{kn+1}\qquad
 (z\in \Delta).
\end{equation}
Next we consider
the problem of finding sharp upper bounds for the Fekete-Szeg\"{o}
coefficient functional associated with the $k$th root transform for
functions in the class $\mathcal{SK}(\alpha)$.

\begin{theorem}\label{t3.1}
Let that $f\in\mathcal{SK}(\alpha)$, $-3<\alpha\leq1$ and $F$ is the
$k$th root transform of $f$ defined by \eqref{F(z)}. Then, for any
complex number $\mu$,
\begin{equation}\label{1t31}
 \left|b_{2k+1}-\mu
 b_{k+1}^2\right|\leq\frac{3-\alpha}{6k}\max\left\{1,\left|\frac{2\mu-1}{k}\left(1-\frac{\alpha}{3}\right)
 -\frac{\alpha^2-3\alpha+9}{6(3-\alpha)}\right|\right\}.
\end{equation}
The result is sharp.
\end{theorem}
\begin{proof}
Since $f\in\mathcal{SK}(\alpha)$, from Definition \ref{Def} and definition of
subordination, there exists $w\in\mathfrak{B}$ such that
\begin{equation}\label{2t31}
 zf'(z)/f(z)=\widetilde{q}_{\alpha}(w(z)).
\end{equation}
We now define
\begin{equation}\label{3t31}
 p(z)=\frac{1+w(z)}{1-w(z)}=1+p_1z+p_2z^2+\cdots.
\end{equation}
Since $w\in\mathfrak{B}$, it follows that $p\in\mathcal{P}$. From \eqref{sumq} and
\eqref{3t31} we have:
\begin{equation}\label{4t31}
 \widetilde{q}_{\alpha}(w(z))=1+\frac{1}{2}\mathcal{B}_1p_1z+\left(\frac{1}{4}\mathcal{B}_2p_1^2+
 \frac{1}{2}\mathcal{B}_1\left(p_2-\frac{1}{2}p_1^2\right)\right)z^2+\cdots.
\end{equation}
Equating the coefficients of $z$ and $z^2$ on both sides of
\eqref{2t31}, we get
\begin{equation}\label{a2}
a_2=\frac{1}{2}\mathcal{B}_1p_1,
\end{equation}
and
\begin{equation}\label{a3}
 a_3=\frac{1}{8}\left(\mathcal{B}_1^2+\mathcal{B}_2\right)p_1^2+\frac{1}{4}\mathcal{B}_1\left(p_2-\frac{1}{2}p_1^2\right).
\end{equation}
A computation shows that, for $f$ given by \eqref{f},
\begin{equation}\label{FF(z)}
 F(z)=[f(z^{1/k})]^{1/k}=z+\frac{1}{k}a_2z^{k+1}+\left(\frac{1}{k}a_3-\frac{1}{2}\frac{k-1}{k^2}a_2^2\right)z^{2k+1}+\cdots.
\end{equation}
From equations \eqref{F(z)} and \eqref{FF(z)}, we have
\begin{equation}\label{bk}
 b_{k+1}=\frac{1}{k}a_2\quad {\rm and}\quad b_{2k+1}=\frac{1}{k}a_3-\frac{1}{2}\frac{k-1}{k^2}a_2^2.
\end{equation}
Substituting from \eqref{a2} and \eqref{a3} into \eqref{bk}, we
obtain
\begin{equation*}
 b_{k+1}=\frac{1}{2k}\mathcal{B}_1p_1,
\end{equation*}
and
\begin{equation*}
 b_{2k+1}=\frac{1}{8k}\left(\mathcal{B}_2+\frac{\mathcal{B}_1^2}{k}\right)p_1^2+\frac{1}{4k}\mathcal{B}_1\left(p_2-\frac{1}{2}p_1^2\right),
\end{equation*}
so that
\begin{equation}\label{b2k-bk}
 b_{2k+1}-\mu
 b_{k+1}^2=\frac{\mathcal{B}_1}{4k}\left[p_2-\frac{1}{2}\left(\frac{2\mu-1}{k}\mathcal{B}_1-\frac{\mathcal{B}_2}{\mathcal{B}_1}+1\right)p_1^2\right].
\end{equation}
Letting
\begin{equation*}
 \mu'=\frac{1}{2}\left(\frac{2\mu-1}{k}\mathcal{B}_1-\frac{\mathcal{B}_2}{\mathcal{B}_1}+1\right),
\end{equation*}
the inequality \eqref{1t31} now follows as an application of Lemma
\ref{FEK} and inserting $\mathcal{B}_1=(3-\alpha)/3$, $\mathcal{B}_2=(\alpha^2-3\alpha+9)/18$. It is easy to check that the result is sharp for the
$k$th root transforms of the function
\begin{equation}\label{fsharp}
  f(z)=z\exp\left(\int_{0}^{z}\frac{\widetilde{q}_{\alpha}(w(t))}{t}dt\right).
\end{equation}
\end{proof}

Putting $k=1$ in Theorem \ref{t3.1}, we have:
\begin{corollary}\label{c3.1}
(Fekete-Szeg\"{o} inequality) Suppose that $f\in\mathcal{SK}(\alpha)$. Then, for any
complex number $\mu$,
\begin{equation}
 \left|a_3-\mu
 a_2^2\right|\leq\frac{3-\alpha}{6}\max\left\{1,\left|(2\mu-1)(1-\alpha/3)-\frac{\alpha^2-3\alpha+9}{6(3-\alpha)}\right|\right\}.
\end{equation}
The result is sharp.
\end{corollary}

If we take $k=1$ and $\alpha=-1$ in Theorem \ref{t3.1}, we get:
\begin{corollary}
  Let $f$ given by the form \eqref{f} be starlike function. Then
  \begin{equation}
 \left|a_3-\mu
 a_2^2\right|\leq\max\left\{2/3,\left|8(2\mu-1)/9-7/12\right|\right\}.
\end{equation}
The result is sharp.
\end{corollary}

Taking $k=1$ and $\alpha=0$ in Theorem \ref{t3.1}, we have:
\begin{corollary}
  Let $f$ given by the form \eqref{f} be in the class $\mathcal{S}^*(1/2)$. Then
  \begin{equation}
 \left|a_3-\mu
 a_2^2\right|\leq\frac{1}{2}\max\left\{1,\left|2\mu-3/2\right|\right\}.
\end{equation}
The result is sharp.
\end{corollary}

It is well known that every function $f\in\mathcal{S}$ has an inverse $f^{-1}$, defined by $f^{-1}(f(z))= z$, $z\in\Delta$ and
\begin{equation*}
  f(f^{-1}(w))=w\qquad (|w|<r_0;\ \ r_0>1/4),
\end{equation*}
where
\begin{equation}\label{f-1}
  f^{-1}(w)=w-a_2w^2+(2a_2^2-a_3)w^3-(5a_2^3-5a_2a_3+a_4)w^4+\cdots.
\end{equation}
\begin{corollary}\label{c3.4}
Let the function $f$, given by \eqref{f}, be in the class $\mathcal{SK}(\alpha)$. Also
let the function $f^{-1}(w)=w+\sum_{n=2}^{\infty}b_nw^n$ be inverse of $f$. Then
\begin{equation}\label{b3}
  |b_3|\leq \frac{3-\alpha}{6}\max\left\{1,\left|\frac{5\alpha^2-33\alpha+45}{6(3-\alpha)}\right|\right\}.
\end{equation}
\end{corollary}
\begin{proof}
  The relation \eqref{f-1} gives
  \begin{equation*}
    b_3=2a_2^2-a_3.
  \end{equation*}
  Thus, for estimate of $|b_3|$, it suffices in Corollary \ref{c3.1}, we put $\mu=2$. Hence
the proof of Corollary \ref{c3.4} is completed.
\end{proof}

\end{document}